\documentclass[12pt,a4paper]{amsart}
\usepackage[all]{xy}
\usepackage[latin1]{inputenc}
\usepackage{amsmath,hyperref}

\newcommand{\C}{{\mathbb C} }
\newcommand{\R}{{\mathbb R} }
\newcommand{\cA}{{\mathcal A} }

\newcommand{\cE}{{\mathcal E} }

\newcommand{\cL}{{\mathcal L} }
\newcommand{\cM}{{\mathcal M} }

\newcommand{\cO}{{\mathcal O} }

\newcommand{\cT}{{\mathcal T} }

\newcommand{\wt}{\widetilde}
\newcommand{\wh}{\widehat}

\def\ol#1{{\overline{#1}}}

\newtheorem{theorem} {Theorem} [section]
\newtheorem{definition}[theorem] {Definition}

\newtheorem{lemma}[theorem]  {Lemma}

\newtheorem{proposition}[theorem] {Proposition}
\newtheorem{corollary}[theorem] {Corollary}

\numberwithin{equation}{section}

\def\pw{vortex--moduli }
\def\ks{Ko\-dai\-ra--Spen\-cer }
\def\ka{K{\"a}h\-ler}

\def\ii{\sqrt{-1}}
\def\ddb{\sqrt{-1}\partial\overline{\partial}}

\def\C{\mathbb{C}}

\def\cinf{C^\infty}
\def\rk{{\mathrm{rk}}}
\def\tr{{\mathrm{tr}}}
\def\ch{{\mathrm{ch}}}

\def\ab{{\alpha\ol\beta}}
\def\ba{{\ol\beta\alpha}}

\def\eoe{End_{\cO_X}(E_1) \oplus End_{\cO_X}(E_2)}
\def\ho{Hom_{\cO_X}(E_2,E_1)}

\def\dbs{{\ol\partial^*}}
\def\db{{\ol\partial}}

\begin{document}

\title[Coupled vortex equations and Moduli]{Coupled vortex
equations and
Moduli: Deformation theoretic Approach and K\"ahler Geometry}

\author[I. Biswas]{Indranil Biswas}
\address{School of Mathematics, Tata Institute of Fundamental
Research, Homi Bhabha Road, Bombay 400005, India}

\email{indranil@math.tifr.res.in}

\author[G. Schumacher]{Georg Schumacher}

\address{Fachbereich Mathematik und Informatik,
Philipps-Universität Marburg, Lahnberge, Hans-Meerwein-Strasse,
D-35032 Marburg, Germany}

\email{schumac@mathematik.uni-marburg.de}

\subjclass[2000]{Primary 53C07, 14D20; Secondary 14D21}

\keywords{Coupled vortex equations, moduli space,
K\"ahler structure, curvature}

\date{}

\begin{abstract}
We investigate differential geometric aspects of moduli spaces
parametrizing solutions of coupled vortex equations
over a compact K\"ahler manifold $X$. These solutions are known
to be related to polystable triples via a Kobayashi--Hitchin type
correspondence. Using a characterization of infinitesimal
deformations in terms of the cohomology of a certain elliptic double
complex, we construct a Hermitian structure on these moduli spaces.
This Hermitian structure is proved to be Kähler. The proof involves
establishing a fiber integral formula for the Hermitian form. We
compute the curvature tensor of this Kähler form. When $X$ is a
Riemann surface, the holomorphic bisectional curvature turns out to
be semi--positive. It is shown that in the case where $X$ is a smooth
complex projective variety, the K\"ahler
form is the Chern form of a Quillen metric on a certain determinant
line bundle.
\end{abstract}

\maketitle

\section{Introduction}

A holomorphic vector bundle together with a holomorphic section of
it will be called a pair. More generally, a pair of holomorphic vector
bundles on a compact K\"ahler manifold
together with a holomorphic homomorphism between them will be
called a triple. Pairs and triples have been investigated
extensively by Bradlow and Garc\'{i}a-Prada (they introduced these
objects) and others \cite{AG, bra1, b2, g-p0, g-p, sdgw, b-c2, b-g-g}.

The notion of stability of a usual vector bundle generalizes to the
context of triples; the definition of a stable triple is recalled
in Section \ref{sec4.1}. The stable pairs are related, by a
Kobayashi--Hitchin type correspondence, to the solutions of vortex
equation, and the stable triples are known to be related to the
solutions of the coupled vortex equations. The coupled vortex
equations are recalled in Section \ref{sec4.1}, and
the Kobayashi--Hitchin correspondence between the stable triples
and the solutions of the coupled vortex equations is recalled
in Section \ref{sec4.2}.

Fix a compact Kähler manifold $X$ equipped with a Kähler form
$\omega$. Take a triple $(E_1, E_2, \phi)$ over $X$, where $E_1$ and
$E_2$ are holomorphic vector bundles over $X$ and
\[
\phi : E_2\longrightarrow E_1
\]
is a holomorphic homomorphism of vector bundles.
Associated to this triple we have a complex of ${\mathcal
O}_X$--modules
$$
C^\bullet \,:\, 0 \longrightarrow C^0 := End_{\cO_X}(E_1) \oplus
End_{\cO_X}(E_2)
\stackrel {\Delta}{\longrightarrow} C^1 := Hom_{\cO_X}(E_2,E_1)
\longrightarrow 0\, ,
$$
with
$$
\Delta(\psi_1,\psi_2)\,=\, \psi_1\circ \phi - \phi \circ
\psi_2\,=:\,[\psi,\phi]\, ,
$$
where $End_{\cO_X}(E_1) \oplus End_{\cO_X}(E_2)$ is at the $0$--th
position. It is easy to see that the space of all endomorphisms of
the triple $(E_1, E_2, \phi)$ coincides with the $0$--th
hypercohomology ${\mathbb H}^0(C^\bullet)$. The first
hypercohomology ${\mathbb H}^1(C^\bullet)$ parametrizes all
infinitesimal deformations of the triple, while the second
hypercohomology contains the obstructions to deformations.

Since the homomorphism $\phi$ is holomorphic, it intertwines the
Dolbeault resolutions of $E_1$ and $E_2$. Therefore, we obtain a
double complex of $C^\infty$ flabby sheaves which gives a resolution
of $C^\bullet$. Consequently, the cohomologies associated to this
double complex coincide with the hypercohomologies of $C^\bullet$.

Now, if the vector bundles $E_1$ and $E_2$ admit Hermitian
structures that give solutions of the coupled vortex equations for
$(E_1, E_2, \phi)$, then the Hermitian metrics on $E_1$ and $E_2$
and the K\"ahler form $\omega$ together define Hermitian structures
on all the terms of the above mentioned
double complex associated to $C^\bullet$. Consequently, we obtain
harmonic representatives of the hypercohomologies of $C^\bullet$.

This immediately induces an $L^2$ metric on the hypercohomologies of
$C^\bullet$, in particular, on the first hypercohomology. Thus the
infinitesimal deformations of the triple $(E_1, E_2, \phi)$ are
equipped with a Hermitian structure. This gives us a Hermitian
structure on the moduli space of solutions of the coupled vortex
equations. We call this Hermitian structure on the moduli space the
\textit{vortex-moduli} metric.

We prove the following theorem.

\begin{theorem}
The \pw Hermitian metric is actually a K\"ahler metric
in the orbifold sense.
\end{theorem}

The following theorem is used in our investigation of the \pw
metric.

\begin{theorem}\label{th0}
Let $T\,=\,(E_1\, ,E_2\, ,\phi)$ be a
stable triple over a compact \ka\ manifold
$(X\, ,\omega_X)$. Let $S$ be a reduced complex space with a
base point $s_0\in S$, and let $\mathcal T\,=\,(\mathcal E_1\,
,\mathcal E_2\, , \Phi)$ be a deformation of $T$ over $(S\, ,s_0)$.
Let $h_1$ and
$h_2$ be Hermitian metrics on $E_1$ and $E_2$ respectively that
solve the coupled vortex equations for $T$. Then there exists a
neighborhood $U$ of $s_0$ such that the solutions can be extended to
the fibers $\mathcal T_s$, $s\in U$, in a $\cinf$ way.
\end{theorem}

The above theorem says that, in a holomorphic family of stable
triples, a solution of the coupled vortex equations can be extended
in a unique way to the neighboring fibers. As an application of
Theorem \ref{th0}, we get the existence of a moduli space of
solutions of the coupled vortex equations on stable triples in the
category of (not necessarily Hausdorff) complex spaces;
see Section \ref{sec..ms}.

In Theorem \ref{thm.curv.} we compute the curvature of the \pw
metric.

Theorem \ref{thm.curv.} has the following corollary (see
Corollary \ref{corollary1}):

\begin{corollary}
If $X$ is a Riemann surface, then the holomorphic bisectional
curvature of the \pw metric is semi--positive.
\end{corollary}

Under the extra assumption that $X$ is complex projective, we
construct a certain holomorphic Hermitian line bundle over the
moduli space of stable triples whose curvature coincides with the
\pw form. The line bundle in question is a determinant bundle
associated to direct images, and the Hermitian structure is given by
a construction due to Quillen and Bismut--Gillet--Soul\'e.

We now give a very brief description of the contents of the
individual sections.

In Section~\ref{se:basic}, we collect basic definitions and
notations. In Section~\ref{Sect3}, deformations of triples are
studied. In Section~\ref{Sect4} the coupled vortex equations are
investigated, and the deformation theoretic approach introduced in
Section~\ref{Sect3} is pursued further. In Section~\ref{ellcom} the
associated elliptic complex is investigated. In Section~\ref{Sect6},
the \pw metric is constructed. In Section~\ref{Sect7} a fiber
integral formula for this metric is established, and it is shown
that the \pw metric is K\"ahler. In Section~\ref{Sect8} we compute the
curvature of the \pw metric.

\section{Basic definitions}\label{se:basic}

Let $X$ be a compact, connected \ka\ manifold, of complex dimension
$n$,
equipped with a \ka\ form $\omega_X$. We will write
$$\omega_X \,= \,\ii
g_{\alpha\ol \beta} dz^\alpha \wedge dz^\ol\beta
\,=\,\ii \sum_{\alpha ,\beta =1}^n
g_{\alpha\ol \beta} dz^\alpha \wedge dz^\ol\beta
$$
with respect to
local holomorphic coordinates $(z^1,\dots,z^n)$, and we will always
use the summation convention. In the sequel, we identify locally
free coherent analytic sheaves on K\"ahler manifolds
with holomorphic vector bundles on them.

We will use the following conventions. The Kähler form $\omega_X$
gives rise to a connection on $X$, which we will, given
any complex space $S$, extend in a flat
way (in the direction
of $S$) to $X\times S$. As above, we will denote by $z^\alpha,
z^\gamma,\ldots$ holomorphic local coordinates on $X$ together with
the conjugates $z^\ol\beta, z^\ol\delta,\ldots$. We will denote by
$s^i, s^k, \ldots$ and $s^\ol\jmath, s^\ol\ell,\ldots$
respectively similar
coordinates on $S$ if $S$ is smooth. If $S$ is not smooth, then
$s^i, s^k, \ldots$ and $s^\ol\jmath, s^\ol\ell,\ldots$ will denote
coordinates on
an ambient smooth space into which a neighborhood, of $S$, of a
given base point in $S$ is minimally embedded. We use
the semi--colon
notation (;) for covariant derivatives of sections, and also of
differential forms or tensors, with values in the
holomorphic Hermitian vector
bundles with respect to connections
induced by the K\"ahler metric on $X$ and the Hermitian
connection on the holomorphic Hermitian vector bundles.
Let the Hermitian connection $\theta_E$ on any
holomorphic vector bundle $E$
over $X$ be given locally by matrix--valued
$(1,0)$--forms $\{\theta_\alpha\}_{\alpha=1}^n$ with respect to some
local trivialization of $E$. Let $\sigma$ be a locally defined
section of $E$, which is a vector--valued function with respect to
the trivialization of $E$. We use
$$
\frac{\partial \sigma}{\partial z^\alpha}\,=\,
 \partial_\alpha \sigma \,=\, \sigma_{|\alpha}\, ,
$$
and set
$$
\sigma_{;\alpha} \,=\, \nabla_\alpha \sigma \,=\, \sigma_{|\alpha} +
\theta_\alpha\circ\sigma
$$
and
$$
\sigma_{;\ol\beta}\,=\,\sigma_{|\ol\beta}\, .
$$
Hence
$$
\sigma_{;\ab} \,=\, \sigma_{;\ba} -  R_\ab\circ \sigma\, ,
$$
where $R_\ab$ denote the components of the curvature form
$\Omega_\ab\,= \,\theta_{\alpha|\ol\beta}$ of the
connection $\theta_E$. For {\it tensors} with
values in the endomorphism bundle
$End_{{\mathcal O}_X}E$, we also have the contributions
that arise from the Kähler connection on the base. For any
differentiable homomorphism of vector bundles
$$
\psi\,:\, E_2 \,\longrightarrow \,E_1\, ,
$$
where $(E_i,h_i)$ are holomorphic Hermitian vector bundles, with
curvature tensors $R^i_\ab$, we have
$$
\psi_{;\ab}\, =\, \psi_{;\ba} -  R_\ab^1\circ \psi +
\psi\circ R_\ab^2\, ,
$$
and we will write
\begin{equation}\label{conv}
\psi_{;\ab} \,=\, \psi_{;\ba} -  [R_\ab\, ,\psi]
\end{equation}
for short.

\section{Deformations}\label{Sect3}

Let $E_1$ and $E_2$ be  holomorphic vector bundles, and let
$$
\phi\,:\,E_2 \,\longrightarrow \,E_1
$$
be an ${\mathcal O}_X$--linear homomorphism. By definition, an
automorphism of the triple $$T\,=\,(E_1\, , E_2\, , \phi)$$ consists of
a pair of automorphisms $\psi_1$ and $\psi_2$, of $E_1$ and $E_2$
respectively, such that
$$
\phi \circ \psi_2 \,=\, \psi_1 \circ \phi\, .
$$

A holomorphic family of such triples over a complex
parameter space $S$ consists of a triple
$$
\cT\,=\,(\cE_1\, ,\cE_2\, ,\Phi)
$$ on $X\times S$. For any point
$s\in S$, the fiber $\cT_s$ is just the restriction of
$(\cE_1,\cE_2,\Phi)$ to $X\times \{s\}\simeq X$. Using the notion
of a holomorphic family, we can derive the notion of a deformation
of an object $T$ over a space $(S,s_0)$ with a distinguished base
point $s_0\in S$
in the usual way, which is done by fixing an isomorphism $T
\stackrel{\sim}{\longrightarrow}\cT_{s_0}$. Isomorphism classes of
deformations of such triples $(E_1,E_2,\phi)$ satisfy the
Schlessinger condition \cite{Schless},
and semi--universal deformations exist by the general theory.

Isomorphism classes of infinitesimal deformations of
$(E_1,E_2,\phi)$ over the double point $$D\,= \,\C[\varepsilon]/
\varepsilon^2 \,=\, (\C \oplus \varepsilon
\C\, , 0)$$
with $\varepsilon^2\,=\,0$ can be identified with the equivalence
classes of extensions of the homomorphism
$$\phi\,:\,E_2 \,\longrightarrow\, E_1
$$
by itself, i.e.,
the equivalence classes of diagrams of the following type:
\begin{equation}\label{ext}
\xymatrix{0 \ar[r] & \varepsilon E_2 \ar[d]^\phi\ar[r]& \cE_2
\ar[r]\ar[d]^\Phi & E_2 \ar[r] \ar[d]^\phi & 0
\\ 0 \ar[r] &\varepsilon E_1 \ar[r] & \cE_1  \ar[r] & E_1 \ar[r] & 0 &}
\end{equation}
Here $\varepsilon E_j \,\hookrightarrow \, \cE_j$,
$j\,=\,1\, ,2$, refers to the
$\cO_D$--module structure of the $\cE_j$ (cf.\ \cite{biswas-ramanan,
b-g-g}).

In order to describe infinitesimal deformations and infinitesimal
automorphisms respectively of such triples $(E_1,E_2,\phi)$, we use
the following complex of $\cO_X$--modules:
\begin{equation}\label{c0}
C^\bullet \,:\, 0 \longrightarrow C^0 :=
End_{\cO_X}(E_1) \oplus End_{\cO_X}(E_2) \stackrel
{\Delta}{\longrightarrow} C^1 := Hom_{\cO_X}(E_2,E_1)
\longrightarrow 0\, ,
\end{equation}
where
$$
\Delta(\psi_1\, ,\psi_2)\,=\, \psi_1\circ \phi - \phi \circ
\psi_2\,=:\,[\psi\, ,\phi]\, ,
$$
and $End_{\cO_X}(E_1) \oplus End_{\cO_X}(E_2)$ is in the $0$--th
position. The hypercohomology of the complex $C^\bullet$
can be computed from the short exact sequence
$$
0 \,\longrightarrow\, A^\bullet \,\longrightarrow\, C^\bullet\,
\longrightarrow\,B^\bullet \,\longrightarrow\, 0\, ,
$$
where
$$
A^\bullet\,:\, 0 \,\longrightarrow\, 0 \,\longrightarrow\,
Hom_{\cO_X}(E_2,E_1) \,\longrightarrow\, 0
$$
and
$$
B^\bullet\,:\, 0\,\longrightarrow\, End_{\cO_X}(E_1) \oplus
End_{\cO_X}(E_2)\to 0 \,\longrightarrow\, 0\, .
$$
So $A^1 \,=\, C^1$ and $B^0 \,=\, C^0$.
We have a long exact sequence of hypercohomologies
\begin{gather}
0 \,\longrightarrow\, \mathbb H^0(C^\bullet)
\,\longrightarrow\, H^0(X,\eoe)
 \,\stackrel{\Delta}{\longrightarrow}\, H^0(X,\ho) \notag \\
\quad \,\longrightarrow\,  \mathbb H^1(C^\bullet)
\,\longrightarrow\, H^1(X,\eoe)\,
\stackrel{\Delta}{\longrightarrow}\, H^1(X,\ho) \\
\,\longrightarrow\, \mathbb H^2(C^\bullet)
\,\longrightarrow\, \ldots \hspace{9cm}
\notag
\end{gather}
The hypercohomology groups $\mathbb H^0(C^\bullet)$ and $\mathbb
H^1(C^\bullet)$ parametrize respectively the endomorphisms and the
infinitesimal deformations of $(E_1,E_2,\phi)$, whereas the
obstructions of infinitesimal deformations live in $\mathbb
H^2(C^\bullet)$.

Let
$$
\cT \,\longrightarrow\, S
$$
with $\cT\,=\,(\cE_1\, ,\cE_2\, ,\Phi)$ be a holomorphic family of
triples $\cT_s$, $s\in S$, parametrized by $S$. Suppose that $h_i$,
$i\,=\, 1\, ,2$, is a Hermitian metric on $\cE_i$ such that the
restrictions
of $h_1$ and $h_2$
to $X\times \{s\}$ are solutions of the coupled vortex equations;
the coupled vortex equations are recalled in \eqref{vortex1} and
\eqref{vortex1a}.

Let $\Omega^i$ be the curvature form of the Hermitian connection for
$h_i$ on $\cE_i$, $i\,=\,1\, ,2$, over $X \times S$ with curvature
tensor
$R^i$. So the contractions
$$
\Omega^i\; \llcorner \frac{\partial}{\partial s_k}
~\,~\,~\,\text{~coincides~\,with~}~\,~\,~\,
R^i_{k\ol\beta}\, dz^\ol\beta\, .
$$

\section{Stability, coupled vortex equation and Kobayashi--Hitchin
correspondence}\label{Sect4}

\subsection{Notions}\label{sec4.1}
Let $$
T\,=\, (E_1\, ,E_2\, ,\phi)
$$
be a triple. A
sub--triple $T'$ of $T$
consists of coherent torsionfree subsheaves $E_i' \subset
E_i$, $i\,=\, 1\, ,2$, such that $\phi'\,:=\,\phi|E_2'$ maps $E_2'$ to
$E_1'$. A sub--triple is called proper if it is not equal to $T$.

For a real number $\alpha$, the $\alpha$--degree and
$\alpha$--slope of $T$ are defined as follows:
\begin{eqnarray}\label{alphadeg}
  \deg_\alpha(T)&:=& \deg(E_1)+ \deg(E_2) + \alpha \cdot\rk(E_2)\\
  \mu_\alpha(T) &:=& \frac{\deg_\alpha(T)}{\rk E_1
+ \rk E_2}.\label{alphadega}
\end{eqnarray}
The degree of a coherent analytic sheaf on $X$ is defined in terms of the
Kähler metric $\omega_X$.

The $\alpha$--degree and $\alpha$--slope of a subtriple of $T$ is
defined exactly as done in \eqref{alphadeg} and \eqref{alphadega}
respectively.

\begin{definition}\label{def-tr.s}
A triple $T \,=\,(E_1\, ,E_2\, ,\phi)$ is called
$\alpha$--stable if for any nonzero proper sub--triple $T'
\,=\, (E'_1,E'_2,\phi')$, with $ \rk E'_1 + \rk E'_2 \,<\,
\rk E_1 + \rk E_2$, the inequality
$$
\mu_\alpha(T')\,<\, \mu_\alpha(T)
$$
holds.

If the weaker inequality $\mu_\alpha(T')\,\leq\, \mu_\alpha(T)$ holds
for any non--zero proper  subtriple, then $T$ is called
$\alpha$--semistable.

An $\alpha$--semistable triple is called $\alpha$--polystable if it is
a direct sum of $\alpha$--stable triples.
\end{definition}

Let $h^i$ be Hermitian metrics on $E_i$ with curvature forms
$\Omega^i$, where $i\,=\,1\, ,2$. Assume that the $\omega_X$--volume of
$X$
is normalized to $2 \pi$. Denote by $\Lambda_X$ the operator dual to
the exterior multiplication by $\omega_X$ of differential forms with
values in vector bundles.

\begin{definition} The coupled vortex equations for $((E_1,h_1),
(E_2,h_2),\phi)$ read as
\begin{eqnarray}\label{vortex1}
  \ii \Lambda_X \Omega^1 + \phi\phi^* &=& \tau_1 \cdot{\rm Id}_{E_1}\\
  \ii \Lambda_X \Omega^2 - \phi^*\phi &=& \tau_2 \cdot
  {\rm Id}_{E_2}\label{vortex1a}
\end{eqnarray}
or equivalently
\begin{eqnarray}\label{vortex2}
  g^\ba R^1_\ab + \phi\phi^* &=& \tau_1 \cdot {\rm Id}_{E_1}\\
  g^\ba R^2_\ab - \phi^*\phi &=& \tau_2 \cdot
  {\rm Id}_{E_2}\label{vortex2a}
\end{eqnarray}
where $\tau_1$ and $\tau_2$ are some real numbers.
\end{definition}

\subsection{Kobayashi--Hitchin correspondence}\label{sec4.2}
For any Hermitian metrics $h_1$ and $h_2$ on $E_1$ and $E_2$
respectively satisfying \eqref{vortex1} and \eqref{vortex1a},
integrating traces of the equations over $X$ yields
\begin{equation}\label{chern_normalize}
  \deg E_1 + \deg E_2 \,=\, \tau_1 \rk E_1 + \tau_2 \rk E_2\, .
\end{equation}

We recall the \textit{Kobayashi--Hitchin correspondence} for triples.
It states that a triple $T \,=\, (E_1\,
,E_2\, ,\phi)$ is $\alpha$--polystable
(see Definition \ref{def-tr.s}) if and only if the following holds:
\begin{itemize}
\item $E_1$ and $E_2$ admit Hermitian metrics satisfying
\eqref{vortex1}
and \eqref{vortex1a} with $\alpha\,=\,\tau_1-\tau_2$, and
\item \eqref{chern_normalize} holds.
\end{itemize}

The Kobayashi--Hitchin correspondence is proved in \cite[p. 182,
Theorem 3.1]{AG}. We also note that in \cite{g-p}, the
Kobayashi--Hitchin correspondence was proved under the assumption
that $\rk E_2 \,=\,1$ with $X$ being an arbitrary compact K\"ahler
manifold, while in \cite{b-c2} it was proved under the assumption
that $\dim X \,=\,1$.

\subsection{Deformation theoretic approach}

All automorphisms of a stable triple
$$
T \,=\, (E_1\, ,E_2\, ,\phi)
$$
are
of the form $\lambda (\text{Id}_{E_1}\oplus \text{Id}_{E_2})$,
where $\lambda\in {\mathbb C}^*$. Therefore, any automorphism
of a stable triple
can be extended to neighboring fibers in a holomorphic family
of triples.
So semi--universal deformations are universal. In
this section, we prove the unique extendibility of solutions of the
coupled vortex equations in a holomorphic family.

\begin{theorem}\label{extendvortex}
Let $T\,=\,(E_1\, ,E_2\, ,\phi)$ be a stable triple over a
compact \ka\ manifold $(X,\omega_X)$. Let $S$
be a reduced complex space with a base
point $s_0\in S$, and let
$$
{\mathcal T}\,=\, (\mathcal E_1\, ,\mathcal E_2\, , \Phi)
$$
be a deformation of $T$ over $(S\, ,s_0)$. Let $h_1$ and $h_2$ be
Hermitian metrics on $E_1$ and $E_2$ respectively that
solve the coupled vortex equations for $T$. Then there exists a
neighborhood $U$ of $s_0$ such that the solutions can be extended
to the fibers $\mathcal T_s$, $s\in U$, in a $\cinf$ way.
\end{theorem}

\begin{proof}
We will use an approach, which is slightly different from the usual
one involving the action of the complexified gauge group.

Let $(E,h)$ stand for any Hermitian holomorphic vector bundle. If
$\sigma$ and $\tau$ are sections, then we write any other Hermitian
metric $\wt h$ on $E$ in the following form:
\begin{equation}\label{wth}
\wt h(\sigma\, ,\tau)\,=\,h(\psi\sigma\, ,\tau)\, ,
\end{equation}
where $\psi\in End(E)$ is a differentiable section which is
self--adjoint with respect to $h$, that is,
$$
\psi^*\,=\,\psi\, .
$$
We rephrase \eqref{wth} in local coordinates. Let $\{e_i\}$ be a set
of local frames of $E$ and denote by
$$
\sigma\,=\,\sum_i \sigma^ie_i\, , \; \tau\,=\,\sum_i \tau^i e_i
$$
sections of $E$. Then
$$
h(\sigma,\tau)\,= \, \sum_{i,j}\sigma^i h_{i\ol\jmath}
\ol\tau^\ol\jmath\, ,
$$
$$
\psi(\sigma)\,=\,  \sum_{i,j}\sigma^i \psi_i^{\; k} e_k\, .
$$
(Observe that for compositions of morphisms, the order of the
corresponding matrix multiplications is reversed.) For the induced
connection $\theta\,=\,\theta_\alpha dz^\alpha$ we have the notation
$$
\theta\sigma\,=\,\sigma^i\theta^{\;k}_{i\alpha}e_kdz^\alpha\, ,
$$
where
$$
\theta\,=\, \partial h \cdot h^{-1}\, , ~\text{ i.e., }~
\theta^{\;k}_{i\alpha}\,=\,
h_{i\ol\jmath|\alpha}\cdot h^{\ol\jmath k}\, .
$$
Now \eqref{wth} reads as
$$
\wt h_{i\ol\jmath}\,=\, \psi_i^{\; k} h_{k\ol\jmath}\, .
$$
The induced connections are
$$
\wt \theta \,=\, \theta + (\partial \psi - [\theta,\psi])\cdot \psi^{-1}
\,=\, \theta + (\partial_\theta\psi)\cdot \psi^{-1}\, ,
$$
where $\partial_\theta$ is the covariant exterior derivative. (Here,
in this section, we need to use the matrix notation instead of the
notations of endomorphisms, which accounts for the above sign.)

If $\chi$ is any differentiable endomorphism of $E$, define $\chi^*$
by
$$
\chi^{*k}_{\;\; i}\,=\,
h_{i\ol\ell}\ol{\chi^\ell_{\;j}}h^{\ol\jmath k}\, .
$$
The following identity for the adjoint $\chi^{\wt *}$
of $\chi$ with respect to the Hermitian structure $\wt h$ holds
$$
\chi^{\wt*}\,=\,\psi^{-1}\chi^*\psi\, ,
$$
and $\chi\,=\,\chi^{\wt*}$, if and only if $\psi \chi$ is
self--adjoint with respect to $h$.

Now the curvatures are (again in terms of the holomorphic structure
on $E$)
\[
\wt \Omega\,=\, \db \wt\theta\, ,
\]
and $\Omega\,=\,\db \theta$. We
note that
$$
(\psi\wt\Omega)^*\,=\,\psi\wt\Omega\, .
$$

In order to prove Theorem~\ref{extendvortex} we extend a pair of
Hermitian metrics $(h_1,h_2)$, which solve the coupled vortex
equations for the triple $T$, as differentiable
families of Hermitian metrics $(h_{1,s}, h_{2,s})_{s\in S}$ for
$\{\mathcal T_s\}_{s\in S}$. Applying self--adjoint differentiable
automorphisms $\psi_{1,s}$ and $\psi_{2,s}$ respectively
of $\cE_{1,s}$
and $\cE_{2,s}$, we get Hermitian metrics $\wt h_{1,s}$ and $\wt
h_{2,s}$  depending differentiably upon the parameter $s$.

For the induced curvature forms
$$
\wt\Omega^j\,=\, \Omega^j + \db(\partial_{\theta_j}(\psi_j) \cdot
\psi_j^{-1})\, ;\; j\,=\,1,2
$$
hold for any fixed $s \in S$.

We consider the assignments
$$
F_s\,: \,(\psi_{1,s},\psi_{2,s})\,\longmapsto\,
\Big(\psi_{1,s}( \ii\Lambda_X \wt\Omega^1_s+ \phi_s\phi_s^{\wt*}-
\tau_1 {\rm Id}_{\cE_{1,s}})\, ,
$$
$$
\psi_{2,s}(\ii\Lambda_X \wt\Omega^2_s
-\phi_s^{\wt*}\phi_s- \tau_2 {\rm Id}_{\cE_{2,s}})\, ,
\int_X (\tr\,\psi_{1,s}+\tr\,\psi_{2,s})g\;dV \Big).
$$

(Here, we denote by $\wt*$ the adjoint with respect to $(\wt
h_{1,s}\, ,\wt h_{2,s})$.)

The first two components of $F_s(\psi_{1,s}\, ,\psi_{2,s})$ are
self--adjoint with respect to $(h_{1,s}\, ,h_{2,s})$. We specify the
domain of the $F_s$.

For sufficiently large $k$, and some $0<\alpha'<1$, we have
Banach manifolds
$$
W_s\,=\,
\{(\psi_{1,s},\psi_{2,s})\,\in\,
W^{k+2,\alpha'}(\mathrm{Aut}(\cE_{1,s}))\oplus
W^{k+2,\alpha'}(\mathrm{Aut}(\cE_{2,s}))\,\vert\,
\psi_{1,s}^*\,=\,\psi_{1,s},
\psi_{2,s}^*\,=\,\psi_{1,s}\}
$$
which give rise to a Banach manifold $W$ together with a smooth map
$$
\pi\,:\,W \,\longrightarrow\, S\, ,
$$
and we have
$$
V_s\,=\,\{(\eta_{1,s},\eta_{2,s}) \,\in\,
W^{k,\alpha'}(\mathrm{End}(\cE_{1,s}))\oplus
W^{k,\alpha'}(\mathrm{End}(\cE_{1,s} ));\, \eta_{1,s}^*\,=\,\eta_{1,s}
\, , \eta_{1,s}^*\,=\,\eta_{2,s}\}
$$
inducing a morphism $\nu: V \to S$ of Banach manifolds. (Here, we
assume for simplicity that $S$ is smooth. However, in case of a
reduced base space $S$ we get local submersions $\pi$ and $\nu$, and
all arguments can be carried over.) The above maps
$F_s$ give rise to a diagram
\begin{equation}\label{VWdiagram}
\xymatrix{W \ar[rr]^F \ar[rd]^\pi && V \times \mathbb R  \ar[ld]^\nu \\
&S \, .&}
\end{equation}

\begin{lemma}\label{le.th}
After restricting $F$ to suitable neighborhoods of $$({\rm
Id}_{E_1}\, ,{\rm Id}_{E_2})\,\in\, W$$
and of $(0,0,0)\,\in\, V\times \mathbb
R$ respectively, the map $F$ in \eqref{VWdiagram} is an
isomorphism of $\,W$ onto a Banach
submanifold $F(W) \,\subset\, V\times \mathbb R$ of codimension one.
\end{lemma}

\begin{proof}
We set $E_j\,:=\,\cE_{j,s_0}$ for $j\,=\,1\, ,2$, and $W_0\,:=
\,W_{s_0}$ as well
as $V_0\,:=\,V_{s_0}$ and otherwise drop the index $s_0$ from now on.
Let $t$ be a complex parameter and
$$
\psi(t)\,=\,(\psi_1(t),\psi_2(t))\,\in\, W_0
$$
a differentiable curve with
$$
\psi(0)\,=\,(\psi_1(0),\psi_2(0))\,=
\,({\rm Id}_{E_1},{\rm Id}_{E_2})\, .
$$
We
compute the derivative of $F_0(\psi(t))$.

We have in matrix notation
$$
\phi^{*}\,=\, h_1\ol\phi^th_2^{-1}\, ,~\; \phi^{\wt*}\,=\,
{\wt h}_1 \ol\phi^t{\wt h}_2^{-1}\,=\,\psi_1\phi^*\psi_2^{-1}.
$$
We return to the endomorphism notation, and get
$$
\left.\frac{d}{dt}\phi^{\wt *}\right|_{t=0} \,= \, \phi^*\dot{\psi_1}
-\dot{\psi_2}\phi^*,
$$
where the dot stands for the $t$--derivative at $t\,=\,0$. Furthermore
$$
\left.\frac{d}{dt}\ii\Lambda\wt\Omega^j\right|_{t=0}\,=\,
\left.\frac{d}{dt}\ii\Lambda(\Omega^j +
\db(\partial_{\theta_j}(\psi_j) \cdot\psi_j^{-1}))
\right|_{t=0}\,=\,
\partial_{\theta_j}^*\partial_{\theta_j}\dot{\psi_j}
$$
since $\partial_{\theta_j}(\text{Id}_{E_j})\,=\,0$. Altogether
the derivative $$D F_0\,:\,T_{\text{Id}} W_0
\,\longrightarrow\, T_0 V_0 \oplus \mathbb R$$ is
given by
\begin{gather}\label{tangent}
DF_0(\chi_1,\chi_2)\,= \hspace{11.5cm}  \\
\left(\partial_{\theta_1}^*\partial_{\theta_1}\chi_1 +
\phi(\phi^*\chi_1 - \chi_2\phi^*)\, ,
\partial_{\theta_2}^*\partial_{\theta_2}\chi_2 -
(\phi^*\chi_1 - \chi_2\phi^*)\phi\, , \int_X (\tr \chi_{1,s}+\tr
\chi_{2,s})g\;dV \right).\nonumber
\end{gather}
We will see that this map is an injection onto $V^0_0+ \R$, where
$V^0_0 \,\subset\, T_0V_0$ denotes the sum of spaces of trace free
endomorphisms.

Suppose that $(\chi_1,\chi_2)$ is in the kernel of $DF_0$. Then
\begin{eqnarray*}
\|\partial_{\theta_1}\chi_1 \|^2 +  \langle \phi^*\chi_1
-\chi_2\phi^*,\phi^*\chi_1  \rangle &=&0 \\
\|\partial_{\theta_2}\chi_2 \|^2 - \langle
\phi^*\chi_1-\chi_2\phi^*, \chi_2\phi^* \rangle &=&0
\end{eqnarray*}
which implies that
$$
\|\partial_{\theta_1}\chi_1 \|^2 + \|\partial_{\theta_2}\chi_2 \|^2
+\|\phi^*\chi_1 - \chi_2 \phi^* \|^2 \,=\,0\, .
$$
In particular the endomorphisms
$\chi_j$ are parallel, and self--adjoint, hence
holomorphic. Furthermore $\chi_1\phi\,=\,\phi\chi_2$. So the pair
$$\chi\,=\, (\chi_1\, ,\chi_2)$$ defines a holomorphic endomorphism of
the
given stable triple. Hence $\chi$ is a real multiple of the
identity. Since the third component in \eqref{tangent} equals zero,
we conclude that $\chi$ must be zero. By Hodge theory the rest follows.
\end{proof}

We return to the proof of Theorem \ref{extendvortex}. Denote by
$\{0\}\times S$ the zero section of
$$
\pi\,:\,V\,\longrightarrow\, S\, .
$$
We know that
$F(W_{s_0})$ intersects $\mathbb R\times \{s_0\}\times\{0\}$
transversally at the origin; here $\mathbb R$ is considered as a
subset of $V_0$. By the above Lemma~\ref{le.th}, the image $F(W)$
intersects  $\mathbb R\times \{s_0\}\times\{0\}$ transversally in a
differentiable section of
$$
\nu\,:\, V \,\longrightarrow\, S\, ,
$$
whose pull--back under $F$
is the desired solution.
\end{proof}
The deformation theoretic approach for the usual moduli space of
irreducible Hermite--Einstein connections can be found in
\cite{f-s1}.

\subsection{Moduli spaces}\label{sec..ms}
As a consequence of Theorem~\ref{extendvortex}, the moduli space of
$\alpha$--stable triples exists in the {\it category of reduced,
complex spaces, which are not necessary Hausdorff}. It should be
emphasized that the approach of \cite{AG} actually
gives a construction of the
moduli space of $\alpha$--stable triple as well as that of
irreducible solutions of the coupled vortex equations; we recall
that in \cite{AG} the
dimension reduction techniques are employed. More
precisely, from \cite{AG} it follows that a moduli space of
$\alpha$--stable triples on a compact K\"ahler manifold $X$ is
realized as an analytic subspace of a moduli space of stable vector
bundles on $X\times{\mathbb P}_1$. Similarly, a moduli space of
solutions of the coupled vortex equations on $X$ is realized as an
analytic subspace of a moduli space irreducible solutions of the
Hermite--Einstein equation on $X\times {\mathbb P}_1$.

\begin{theorem}\label{existmod}
Given a compact \ka\ manifold $X$, the moduli space of
objects of the form $((E_1\, ,h_1)\, , (E_2\, ,h_2)\, ,\phi)$,
where $(E_1\, , E_2\, ,\phi)$ is a stable triple over
$X$, and $h_i$, $i\,=\, 1\, ,2$, are Hermitian structures
on $E_i$ satisfying the coupled vortex equations, exists.
\end{theorem}

\section{Elliptic complex}\label{ellcom}

Take any triple $(E_1, E_2,\phi)$. As before, let $C^\bullet$ be the
complex associated to it (see \eqref{c0}).

In order to use the theory of elliptic operators, we observe that
the Dolbeault complexes provide a resolution $ C^{\bullet\bullet}$
of $C^\bullet$:
\begin{equation}\label{dolbres}
\begin{xymatrix}
  {0 \ar[r] & \eoe \ar[r]\ar[d]^\Delta &
  \cA^{0,0}(\eoe) \ar[r]\ar[d]^\Delta & \ldots \\
  0 \ar[r] & \ho \ar[r] &  \cA^{0,0}(\ho) \ar[r] & \ldots  &
  }
\end{xymatrix}
\end{equation}
$$
\begin{xymatrix}
  {\cA^{0,i}(\eoe) \ar[r]\ar[d]^\Delta &
  \cA^{0,i+1}(\eoe) \ar[r]\ar[d]^\Delta & \ldots \\
  \cA^{0,i}(\ho) \ar[r] &  \cA^{0,i+1}(\ho) \ar[r] & \ldots  &
  }
\end{xymatrix}
$$
where $C^{0,i}\,=\,\cA^{0,i}(\eoe)$ and $C^{1,i}\,=\,
\cA^{0,i}(\ho)$.

It follows immediately that $\Delta\circ \db\, =\, \db \circ
\Delta$.

Let ${\wt C}^\bullet$ be the complex which is constructed in the
following way. Define
$$
{\wt C}^i\, :=\, C^{0,i} \oplus C^{1,i-1}
$$
with the convention that $C^{1,-1}\, =\, 0$. The
homomorphisms
$$
(\overline{\partial}\, ,(-1)^i\Delta)\, :\, C^{0,i}\,
\longrightarrow\,C^{0,i+1} \oplus C^{1,i}
$$
and
$$
(0\, ,\overline{\partial})\, :\,  C^{1,i-1} \, \longrightarrow\,
C^{0,i+1} \oplus C^{1,i}
$$
together give the homomorphisms ${\wt C}^i\, \longrightarrow\,
{\wt C}^{i+1}$ that define the complex ${\wt C}^\bullet$.

We note that
the cohomology of $C^\bullet$ can be identified with the
cohomology of the single complex $\wt C^\bullet$ associated to
$C^{\bullet\bullet}$.

Following the construction in \cite{s-t} one can see that the global tensors
$\Omega^i$ and $\Phi$ over $X \times S$ already describe the infinitesimal
deformations. (This is in fact true for any pair of Hermitian metrics.) In
other words, we have the following lemma:
\begin{lemma}\label{le:kod-sp}
 Let
  $$
  \rho_{s_0}\,:\, T_{s_0}S \,\longrightarrow\, \mathbb H^1(\wt
C^{\bullet\bullet})
  $$
  be the Kodaira--Spencer mapping. Then
  \begin{gather*}
 \mu_i\,=\,\left(-\Phi_{;i}, (R^1_{i\ol\beta}dz^\ol\beta,
R^2_{i\ol\beta}dz^\ol\beta) \right)
  \hspace{8cm}\\
\in\, \cA^{0,0}(X,\ho)\oplus \cA^{0,1}(X,\eoe)
 \end{gather*}
represents the class
  $$
  \rho_{s_0}\left(\left.\frac{\partial}{\partial
  s_i}\right|_{s=s_0}\right)\in \mathbb H^1(\wt C^{\bullet\bullet})\, .
  $$
 \end{lemma}
Given Hermitian metrics $h^i$ on $E_i$, $i\,=\,1\, ,2$, together
with $\omega_X$, the single complex
associated with \eqref{dolbres} is {\it elliptic}.

We consider the complex
$$
\Gamma(\wt C^\bullet): 0\longrightarrow C^{0,0}(X)
\stackrel{d^0}{\longrightarrow} C^{1,0}(X)\oplus C^{0,1}(X)
\stackrel{d^1}{\longrightarrow} C^{1,1}(X) \oplus C^{0,2}(X)
\longrightarrow\ldots
$$
More precisely,
\begin{gather*}
\Gamma(\wt C^\bullet): 0\,\longrightarrow\, \cA^{0,0}(X,\eoe)
\stackrel{d^0}{\longrightarrow}\hspace{7cm}  \\
\cA^{0,0}(X,\ho)\oplus \cA^{0,1}(X,\eoe) \stackrel{d^1}
{\longrightarrow}\hspace{2cm} \\ \cA^{0,1}(X,\ho) \oplus
\cA^{0,2}(X,\eoe ) \stackrel{d^2}{\longrightarrow} \ldots\\
\cA^{0,i-1}(X,\ho) \oplus
\cA^{0,i}(X,\eoe ) \stackrel{d^i}{\longrightarrow} \ldots\\
\cA^{0,n-1}(X,\ho) \oplus
\cA^{0,n}(X,\eoe ) \stackrel{d^n}{\longrightarrow}\\
\cA^{0,n}(X,\ho){\longrightarrow}\,0
\end{gather*}
with
\begin{eqnarray}\label{eq:dop}
d^0(f)&=&(\Delta f,\db^0 f)\; ; \qquad f\,=\,(f_1,f_2)  \\
d^1(a,b)&=&(\db^1 a - \Delta b, \db^1 b) \;  ;
\qquad b\,=\,(b_1,b_2)\, .
\end{eqnarray}

The following lemma is evident.

\begin{lemma}\label{le:adjt}
  The adjoint operators $$d^{1*}\,:\, C^1(X) \,\longrightarrow\,
C^0(X)$$ and $d^{2*}\,:\,
  C^2(X)\,\longrightarrow\, C^1(X)$ are given by
  \begin{eqnarray}
  d^{0*}(a,b)&=& (a\phi^*+ \db^*b_1, -\phi^*a + \db^*b_2)\\
  d^{1*}(u,v)&=& (\db^*u , (-u\phi^*+\db^*v_1,\phi^*u+\db^*v_2))
  \end{eqnarray}
\end{lemma}
We return to the situation of Lemma~\ref{le:kod-sp}.
\begin{proposition}
The forms $\mu_i$ are the harmonic representatives of the \ks
classes $\rho(\partial/\partial s|_{s_0})$:

\noindent\parbox{\textwidth}{
\begin{eqnarray}
d\mu_i&=&0 \label{dmu} \\ d^*\mu_i&=&0 \label{dstarmu}
\end{eqnarray}
}
\end{proposition}

\begin{proof}
We know that $d\mu_i\,=\,0$. We refer to Lemma~\ref{le:adjt} for
$d^*\mu_i\,=\,0$ and use the coupled vortex equations \eqref{vortex2}
and \eqref{vortex2a}: The first component equals
$$
(d^*\mu_i)_1\,=\, -\Phi_{;i}\cdot\Phi^*+\dbs
(R^1_{i\ol\beta}dz^{\ol\beta})\,=\,- \Phi_{;i}\cdot \Phi^*- g^\ba
R^1_{i\ol \beta;\alpha}\,=\,-\Phi_{;i}\Phi^*+(\Phi\Phi^*)_{;i}
\,=\, 0\, ,
$$
and the second follows in the same way.
\end{proof}

\section{Hermitian structure on the moduli space}\label{Sect6}

In this section we will construct a \pw Hermitian metric on the moduli
space of $\alpha$--stable triples. The corresponding inner product
is given by a natural Hermitian metric on the base of a universal
deformation. Then it will be shown that the \pw Hermitian metric is
actually a Kähler metric. The elliptic complex from
Section~\ref{ellcom} plays a key role in the construction.

Take an effective holomorphic family of holomorphic triples
such that each holomorphic triple in the family is
equipped with Hermitian structures satisfying
the coupled vortex equations. For such that family we will
construct a Hermitian structure on the parameter space
in a functorial way. This construction of the
Hermitian structure on a parameter space would give a
Hermitian structure on the moduli space
of $\alpha$--stable triples.

Let $(s_i)$ be holomorphic coordinates on $S$ around a point $s_0$
in the sense of Section~\ref{se:basic}.

Now, we are in a position to introduce a \pw metric on the parameter
space $S$ for a family of stable triples. The \pw metric is an inner
product $G^{VM}$ on the tangent spaces $T_sS$ of the bases of
holomorphic families, which is positive definite for effective
families, and it is defined in terms of the tensors $\mu_i$
representing the \ks\ classes. (The superscript ``VM'' stands
for ``vortex--moduli metric''.)
This is possible, also in the case
where $S$ is singular because the family of holomorphic
homomorphisms and the curvature forms for the connection on vector
bundles still exist on the first order infinitesimal neighborhood.
The latter fact follows from the approach described above.

\begin{definition}\label{de:pwherm}
{\rm A
Hermitian structure on the tangent space $T_{s_0}S$ is given by}
\begin{gather}
G^{VM}\left(\left. \frac{\partial}{\partial s^i}\right|_{s_0},
\left. \frac{\partial}{\partial s^\ol\jmath}\right|_{s_0} \right)
\,:=\,
G_{i\ol\jmath}^{VM}\,:=\,
\langle\mu_i,\mu_j\rangle \hspace{5cm} \\
\hspace{2cm}=\, \int_X \tr (\Phi_{;i}\Phi^*_{;\ol\jmath})g\, dV +
\int_X \tr( g^\ba R^1_{i\ol\beta} R^1_{\alpha\ol\jmath}) g\, dV +
\int_X \tr( g^\ba R^2_{i\ol\beta} R^2_{\alpha\ol\jmath}) g\, dV\,\,
. \nonumber
\end{gather}
{\rm We set}
\begin{equation}\label{vmkf}
\omega^{VM}\, =\, \ii G_{i\ol\jmath}^{VM}(s) ds^i \wedge ds^\ol\jmath
\end{equation}
{\rm and call this Hermitian structure the}
vortex--moduli metric.
\end{definition}

\section{Fiber integral and Quillen metric}\label{Sect7}

We will use the following notion of a Kähler space. Let $W$ be a
polydisk together with a Kähler form $\omega_W$, and $S \,\subset \,
W$ a closed reduced analytic subspace. Then all tangent spaces
$T_sS$, $s\in S$, carry an induced Hermitian metric. If $S$ is any
reduced complex space together with a family of Hermitian metrics on
all tangent spaces $T_sS$, $s\in S$, which is locally of the above
kind, then the induced real $(1,1)$--form $\omega_S$ on $S$ is called
a {\it Kähler form}. Observe that Kähler forms on reduced complex
spaces possess local $\partial\overline\partial$--potentials of class
$C^\infty$. Clearly the restriction of a Kähler form on a complex
space to a complex analytic subspace is again a Kähler form.

We begin with a general remark about fiber integrals. For any
projection of differentiable manifolds
$$
Z\times R \,\longrightarrow\, R\, ,
$$
where $Z$
is compact of dimension $m$, the push forward of an $(m+k)$--form
$\chi$ of class $\cinf$ is a $\cinf$--form of degree $k$ given as a
fiber integral
$$
\int_Z \chi \,:=\, \int_{(Z\times R)/R} \chi\, .
$$
This process of fiber integration applied to proper smooth,
holomorphic maps is type preserving in the sense that for fibers of
complex dimension $n$ the fiber integral of any $(n+k,n+\ell)$--form
of class $\cinf$ is a differentiable $(k,\ell)$--form. In our case
the base space $S$ is not necessarily smooth, but the given
$(n+1,n+1)$--forms will be the local restrictions, from a smooth
ambient space, of $\ddb$--exact $\cinf$ forms. This will yield
$(1,1)$--forms with local $\ddb$--potential of class $\cinf$ in the
sense described above. (For details and generalizations see
\cite{va}.)

In place of earlier sub--indices (as in $E_i$), in
the following proposition we use the super--indices in order
not to confuse with the sub--indices for coordinates.

\begin{proposition}\label{pr:fibint}
Denote by $\Omega^\nu$ the curvature form of $(E^\nu,h^\nu)$, $\nu
\,=\,1\, ,2$. Then the following fiber integral formula for
the \pw form holds

\noindent \parbox{\textwidth}{
\begin{equation}\label{eq:fibint}
 \omega^{VM}\,=\,
\frac{1}{2}\sum_{\nu=1,2}\left(\int_X
\tr(\Omega^\nu\wedge\Omega^\nu)\wedge \frac{\omega_X^{n-1}}{(n-1)!}
+ \tau_\nu \int_X \tr \Omega^\nu \frac{\omega^n}{n!}\right) + \ddb
\int_X \tr(\Phi\Phi^*) \frac{\omega_X^n}{n!}
\end{equation}
}
\end{proposition}

\begin{proof}
We compute
\begin{gather*}
\zeta\,=\,\frac{1}{2}\sum_{\nu=1,2}\int_X
\tr(\Omega^\nu\wedge\Omega^\nu)\wedge \frac{\omega_X^{n-1}}{(n-1)!}
\,=\,  - \frac{1}{2}\sum_{\nu=1,2}\int_X
\tr(\ii\Omega^\nu\wedge\ii\Omega^\nu)
\wedge \frac{\omega_X^{n-1}}{(n-1)!}\\
\hspace{2cm} = \,\sqrt{-1}
\sum_{\nu=1,2}\int_X\tr(R^\nu_{\alpha \ol\jmath}\cdot
R^\nu_{i\ol\beta} - R^\nu_\ab \cdot R^\nu_{i\ol\jmath})g^\ba g
\,dV ds^i\wedge ds^\ol\jmath \\ =\, \ii
\left(\sum_{\nu=1,2}
\int_X\tr(R^\nu_{\alpha \ol\jmath}\cdot R^\nu_{i\ol\beta}) + \int_X
\tr(\Phi\Phi^*R^1_{i\ol\jmath}-\Phi^*\Phi R^2_{i\ol\jmath}-\tau_1
R^1_{i\ol\jmath}- \tau_2 R^2_{i\ol\jmath})\,g\, dV \right)
ds^i\wedge ds^\ol\jmath
\end{gather*}
by \eqref{vortex2} and \eqref{vortex2a}.

Now, on $X\times \{s\}$ we have
$$
\tr (\Phi^*(-R^1_{i\ol\jmath} \Phi
+\Phi R^2_{i\ol\jmath}))\,=\, \tr (\Phi^*(\Phi_{;i\ol\jmath} -
\Phi_{;\ol\jmath i}))\,=\,\tr( \Phi^*\Phi_{;i\ol\jmath})
$$
as $\Phi$ is a holomorphic section on the total space. So
\begin{gather*}
\int_X \tr(\Phi\Phi^*R^1_{i\ol\jmath}-\Phi^*\Phi R^2_{i\ol\jmath})\,
g \, dV \,=\, -\int_X \tr(\Phi_{;i\ol\jmath}\Phi^*) \, g \, dV
\hspace{2cm}\\ \hspace{2cm}
=\, \int_X \tr(\Phi_{;i}\Phi^*_{;\ol\jmath})
\, g \, dV - \frac{\partial^2}{\partial s_i \partial s_\ol\jmath}
\int_X \tr( \Phi \Phi^*) \, g \, dV.
\end{gather*}

Furthermore
$$
\int_X \tr (\Omega^\nu) \frac{\omega^n}{n!} \,=\,
\ii \int_X \tr
(R^\nu_{i\ol\jmath}) \, g \, dV ds^i\wedge ds^\ol\jmath\, .
$$
Combining these the proof of the proposition is complete.
\end{proof}

Next, we will express the above in terms of Chern character forms
described as
$$
\ch(\cE,h)\,=\, \sum_{k=0}^n \frac{1}{k!}\tr\left( \vtop{\hbox
{$\underbrace{\frac{\ii}{2\pi}\Omega \wedge\dots \wedge
\frac{\ii}{2\pi}\Omega }$}\hbox{\hspace{15mm}{$k\text{-times}$}}}
\right)
$$
with
$$
\ch_2(\cE,h)\,=\, \frac{1}{2}\left(c_1^2(\cE,h)-2c_2(\cE,h)\right)\, .
$$
In terms of Chern character forms and Chern forms Formula
(\ref{eq:fibint}) reads

\noindent \parbox{\textwidth}{
\begin{eqnarray}
\frac{1}{4\pi^2} \omega_{VM} &=& - \int_X \ch_2(\cE_1
\oplus\cE_2,h^1\oplus h^2)\wedge\frac{\omega_X^{n-1}}{(n-1)!} +\\
\nonumber& & + \frac{\tau_1}{2\pi} \int_X
c_1(\cE,h^1)\wedge\frac{\omega_X^n}{n!}+
\frac{\tau_2}{2\pi} \int_X c_1(\cE_2,h^2)\wedge\frac{\omega_X^n}{n!}\\
\nonumber & &+ \frac{\ii}{4\pi^2}\partial\ol\partial\int_X
\tr(\Phi\Phi^*)\wedge \frac{\omega_X^{n}}{n!}\, .
\end{eqnarray}}

{}From now on, till the end of this section,
we assume that $X$ is a \ka\ manifold whose \ka\ form
is the Chern form
$$
\omega_X\,=\,c_1(\cL\, ,h_\cL)
$$
of a positive Hermitian line bundle $(\cL\, ,h_\cL)$, in particular,
$X$ is a complex projective manifold.

Given a proper, smooth holomorphic map $$f\,:\,\mathcal X
\,\longrightarrow\, S$$ and a
locally free sheaf $\mathcal F$ on $\mathcal X$, the determinant
line bundle of $\mathcal F$ on $S$ is by definition $\det
\underline{\underline R}f_* \mathcal F$ \cite{km, bgs}.

The generalized Riemann--Roch theorem by Bismut, Gillet and Soul\'e
\cite{bgs} applies to Hermitian vector bundles $(\mathcal F\, , h)$ on
$\mathcal X$. It states that the determinant line bundle of
$\mathcal F$ on $S$ carries a Quillen metric, whose Chern form
equals the fiber integral
$$
\int_{\mathcal X/S} {\rm ch}(\mathcal F,h) {\rm td}(\mathcal
X/S,\omega_{\mathcal X})\, ,
$$
where ${\rm ch}(\mathcal F,h)$ and ${\rm td}(\mathcal
X/S,\omega_\mathcal X)$ denote respectively the Chern character form
for $(\mathcal F,h)$ and the Todd character form for the relative
tangent bundle; see \cite[Theorem 0.1]{bgs}, and also \cite{zt} for
$\dim X\,=\,1$.

Let $\cE$ stand for one of the Hermitian vector bundles $\cE_1$ or
$\cE_2$. Let $h$ denote the Hermitian metric on
$\cE$. We first mention
\begin{equation}
\ch(End(\cE))\,= \,r^2 + 2 r \ch_2(\cE) -c_1^2(\cE) + \dots
\end{equation}
where $r$ is the rank of $\cE$, so that for the virtual bundle
$End(\cE)- \cO^{r^2}$ the identity
$$
\ch(End(\cE)- \cO^{r^2}) \,=\, 2r \text{ch}_2(\cE)^2-c_1^2(\cE)+ \dots
$$
holds.

Now
\begin{eqnarray*}
&&\hspace{-8mm} \ch\left((End(\cE),h)\otimes \left( (\cL,h_\cL)-
(\cL^{-1},h_\cL^{-1})\right)^{\otimes (n-1)}\right)  \\
&=& \ch_2\left(End(\cE),h\right) \cdot 2^{n-1} \omega_X^{n-1} +
\ldots\\
&=& \left(2r\left( \frac{1}{2}\tr\left(\frac{\ii}{2\pi}\Omega \wedge
\frac{\ii}{2\pi}\Omega\right)\right) - \left(\tr
\frac{\ii}{2\pi}\Omega \right)^2 \right)2^{n-1}
\omega_X^{n-1} +\dots\\ &=&
2^{n-1}\left(r\cdot\tr\left(\frac{\ii}{2\pi}\Omega\wedge
\frac{\ii}{2\pi}\Omega\right)-\left(\tr
\frac{\ii}{2\pi}\Omega\right)^2\right) \omega_X^{n-1} +\dots
\end{eqnarray*}
The highest exterior power $\Lambda^r \cE$ carries the induced
Hermitian metric $\wh h$, for which the following identity holds:
\begin{eqnarray*}
&& \hspace{-2cm}\ch\left(\left(\left(\Lambda^r \cE,\wh
h\right)-(\Lambda^r \cE,\wh h)^{-1}\right)^{\otimes 2} \cdot \left(
(\cL,h_\cL)-(\cL^{-1},h^{-1})\right)^{\otimes(n-1)}\right)\\
\qquad &=& 2^{n+1} c_1^2(\cE,h)\cdot(\cL,h_\cL)^{n-1}  +\dots \\
&=& 2^{n+1} c_1^2(\cE,h)\cdot \omega_X^{n-1} +\dots \\
&=& 2^{n+1}\left(\tr\frac{\ii}{2\pi}\Omega\right)^2 \omega_X^{n-1} +
\dots
\end{eqnarray*}

Hence we have the following theorem:

\begin{theorem}\label{th.fif}
The \pw metric defined in \eqref{vmkf} has the following expression:
\noindent \parbox{\textwidth}{
\begin{eqnarray*}
&& \hspace{-1cm} \frac{1}{4\pi^2}\omega_{VM} \,=
\\
 &&\sum_{\nu=1,2}\left(-\frac{1}{2^n r_\nu (n-1)!}\int\ch\left(End(\cE_\nu)
\otimes (\cL-\cL^{-1})^{\otimes(n-1)}\right)\right.\\
 &&  - \frac{1}{2^{n+2} r_\nu (n-1)!} \int \ch\left(
(\Lambda^{r_\nu}\cE_\nu-(\Lambda^{r_\nu} \cE_\nu)^{-1})^{\otimes
2}\otimes
(\cL-\cL^{-1})^{\otimes(n-1)}\right)\\
 &&\left. + \frac{\lambda}{2\pi}
\frac{1}{2^n n!}\int \ch\left(\left( \Lambda^{r_\nu}
\cE_\nu-(\Lambda^{r_\nu} \cE_\nu)^{-1}\right)\otimes
(\cL-\cL^{-1})^{\otimes n}\right)\right)\\
 && + \frac{1}{4\pi^2} \ii
\partial\ol\partial \int \tr(\Phi\wedge \Phi^*)\wedge
\frac{\omega_X^n}{n!}\,.
\end{eqnarray*} }

\end{theorem}

Using Theorem \ref{th.fif}, we will express the \pw Kähler form as
the curvature form of a holomorphic Hermitian line bundle.

Let
$$
q\, :\, X\times S \,\longrightarrow\, S
$$
be the canonical projection, where $S$ stands for the base space of
a universal deformation of a stable triple with solution of the
coupled vortex equations. Since our construction is functorial, the
construction descends to the moduli space $\cM$ (after taking
suitable powers of the line bundles on the base).

We introduce the following determinant line bundles $\delta_{j\nu}$
where $j\,=\,1\, ,2\, ,3$ and $\nu \,=\, 1\, ,2$,
equipped with Quillen metrics $h^Q_j$:
\begin{eqnarray*}
\delta_{1\nu} &=& \det \underline{\underline R}q_* \left(
End(\cE_\nu)\otimes
(\cL-\cL^{-1})^{\otimes(n-1)}\right)\\
 \delta_{2\nu} &=& \det
\underline{\underline R}q_* \left(\left(
\Lambda^{r_\nu}\cE_\nu-(\lambda^{r_\nu}\cE_\nu)^{-1}\right)^{\otimes
2} \otimes
(\cL-\cL^{-1})^{\otimes(n-1)}\right) \\
  \delta_{3\nu} &=& \det
\underline{\underline R}q_* \left(\left(
\Lambda^{r_\nu}\cE_\nu-(\lambda^{r_\nu}\cE_\nu)^{-1}\right) \otimes
(\cL-\cL^{-1})^{\otimes n}\right).
\end{eqnarray*}
Setting
$$
\chi\,=\,\int \tr (\Phi\wedge \Phi^*)\wedge \frac{\omega_X^n}{n!}
$$
we equip the trivial bundle $\cO_{\cM_H}$ with the Hermitian metric
$e^\chi$.

Combining Theorem \ref{th.fif} and \cite[Theorem 0.1]{bgs} we have
the following theorem:

\begin{theorem}\label{thm3}
The \pw K\"ahler form is a linear combination of the $(1,1)$--forms
$c_1(\delta_{j\nu},h^Q_{j\nu})$, $j\,=\,1\, ,2\, ,3$, $\nu
\,=\,1\, ,2$, and
$c_1(\cO_{\cM_H},e^\chi)$. For rational $\tau_\nu$, a multiple of
the \pw form is equal to the Chern form of an Hermitian line bundle.
\end{theorem}

We note that
this \pw Kähler metric coincides with the one constructed
in \cite{AG} and \cite{g-p}, where the moduli spaces
of triples and pairs respectively
have been constructed as K\"ahler quotients.

\section{Curvature of the  \pw metric}\label{Sect8}

In this section $X$ will be a compact K\"ahler manifold.

We begin by establishing
a collection of identities for the harmonic \ks
tensors
$$
\mu_i \,\in\, \cA^{0,0}(X,\ho)\oplus \cA^{0,1}(X,\eoe)\, .
$$
In
particular, we need to understand covariant derivatives with respect
to the base directions.

The symmetry

\noindent
 \parbox{\textwidth}{
 \begin{equation}
 \mu_{i;k}\,=\,\mu_{k;i}
 \end{equation}
}

\smallskip\noindent follows immediately from the definition.

We set $R_{i\ol\beta}\,=\,(R_{i\ol\beta}^1,R_{i\ol\beta}^2)$
et cetera.  We
compute the components of $d\mu_{i;k}$: Because of
$$
\Phi_{;ik\ol\beta}\,=\, \Phi_{;i\ol\beta k}-
R^1_{k\ol\beta}\Phi_{;i}+\Phi R^2_{k\ol\beta}\,=
\, (\Phi_{;\ol\beta i} -
R^1_{i\ol\beta}+ R^2_{k\ol\beta})_k - R^1_{k\ol\beta}\Phi_{;i}+\Phi
R^2_{k\ol\beta}
$$
we have
$$
(d\mu_{i;k})_1\,=\,( -\Phi_{;ik\ol\beta}-R^1_{i\ol;k}\Phi+\Phi
R^2_{i\ol\beta;k})dz^\ol\beta\,=\,
(R^1_{i\ol\beta}\Phi_{;k}-\Phi_{;k}R^2_{i\ol\beta}+
R^1_{k\ol\beta}\Phi_{;i}-\Phi_{;i}R^2_{k\ol\beta})dz^\ol\beta\, .
$$
Furthermore, from the identity
$$
R^1_{i\ol\beta;k\ol\delta}\,=\,R^1_{i\ol\beta;\ol\delta
k}-[R^1_{k\ol\delta},R^1_{i\ol\delta}]
$$
it follows that
$$
(d(\mu_{i;k}))_2
$$
$$
\,=\,-(R^1_{i\ol\beta;k\ol\delta}dz^\ol\beta\wedge
dz^\ol\delta\, ,R^2_{i\ol\beta;k\ol\delta}dz^\ol\beta\wedge
dz^\ol\delta)\,=\,
-([R^1_{k\ol\delta},R^1_{i\ol\delta}]\;dz^\ol\beta\wedge
dz^\ol\delta,[R^1_{k\ol\delta}\, ,R^2_{i\ol\delta}]\;dz^\ol\beta\wedge
dz^\ol\delta)\, .
$$
The last term equals
$$
([R^1_{i\ol\beta}dz^\ol\beta,R^1_{k\ol\delta}dz^\ol\delta]\, ,
[R^2_{i\ol\beta}dz^\ol\beta\, ,R^2_{k\ol\delta}dz^\ol\delta])\, ,
$$
(involving a symmetric product of one--forms with values in an
endomorphism bundle).

Now we can introduce an exterior product on $\wt C^\bullet$, in
particular, we define a symmetric exterior product $\wt C^1 \times
\wt C^1 \longrightarrow \wt C^2$:
\begin{equation}
 [\mu_i\wedge \mu_k]\, :=\,
\left(-\Phi_{;i}R^2_{k\ol\beta}+
 R^1_{k\ol\beta}\Phi_{;i}+R^1_{i\ol\beta} \Phi_{;k} -
 \Phi_{;k} R^2_{i\ol\beta}\, , ([R^1_{i\ol\beta} dz^\beta,
 R^1_{k\ol\delta}dz^\ol\delta],[R^2_{i\ol\beta} dz^\beta,
 R^2_{k\ol\delta}dz^\ol\delta])\right)\, ,
 \end{equation}
where $\mu_i\, =\,
(-\Phi_{;i},(R^1_{i\ol\beta}dz^\ol\beta\, ,
 R^2_{i\ol\beta}dz^\ol\beta))$ and
$\mu_k\, =\,(-\Phi_{;k},(R^1_{k\ol\delta}dz^\ol\delta\, ,
 R^2_{i\ol\delta}dz^\ol\delta))$.

Now

 \noindent
 \parbox{\textwidth}{
 \begin{equation}\label{dmuik}
 d(\mu_{i;k})+ [\mu_i\wedge \mu_k]\,=\,0\, .
 \end{equation}
}

Next we compute
$$
d^*(\mu_{i;k})\,=\, (-\Phi_{;ik}\Phi^*-
R^1_{i\ol\beta;k\alpha}g^\ba\, ,
\Phi^*\Phi_{;ik}- R^2_{i\ol\beta;k\alpha}g^\ba).
$$
Because of the coupled vortex equations \eqref{vortex2} and
\eqref{vortex2a} we get

 \noindent
 \parbox{\textwidth}{
 \begin{equation}\label{dstarmuik}
 d^*(\mu_{i;k})\,=\, 0\, .
 \end{equation}
}

We note that
$$
\mu_{i;\ol\jmath} \,=\,( - \Phi_{;i\ol\jmath}\, ,
(R^1_{i\ol\beta;\ol\jmath}dz^\ol\beta,
R^2_{i\ol\beta;\ol\jmath}dz^\ol\beta))\, .
$$
Because of
$$
\Phi_{;i\ol\jmath}\,= \,\Phi_{;\ol\jmath i}- R^1_{i\ol\jmath}\Phi
-\Phi R^2_{i\ol\jmath}\,=\, -\Delta R_{i\ol\jmath}
$$
we have
$$
\mu_{i;\ol\jmath}\,=\, (\Delta R_{i\ol\jmath},\ol\partial
R_{i\ol\jmath})\, ,
$$
which means that

 \noindent
    \parbox{\textwidth}{
        \begin{equation}\label{mudr}
            \mu_{i;\ol\jmath}\,=\, dR_{i\ol\jmath},
        \end{equation}
    }

\noindent where the tensor $R_{i\ol\jmath}$ is considered as a
section of $\wt C^0$.

We can use Hodge theory on the complex $\wt C^\bullet$. So
$$
d^*(\mu_{i;\ol\jmath})\,=\, d^*d R_{i\ol\jmath}
\,= \,\Box R_{i\ol\jmath}\, .
$$

For any section $f\,=\,(f_1\, ,f_2)$ of $\wt C^0$, we compute
\begin{gather*}
d^*d f\,=\, d^*(\Delta f\, ,\db f)\,=\,
( (\Delta f) \Phi^*+ \db^*\db f_1\, ,
-\Phi^* \Delta f + \db^*\db f_2)\\
=\, (f_1\Phi\Phi^* -\Phi f_2 \Phi^*
- g^\ba f_{1;\ba}\, ,- \Phi^* f_1\Phi + \Phi^*\Phi f_2 - g^\ba
f_{2,\ba} )\, . \nonumber
\end{gather*}
We apply the formula for $f\,=\,R_{i\ol\jmath}$. It involves (again by
\eqref{vortex2} and \eqref{vortex2a})
\begin{gather*}
g^\ba R^1_{i\ol\jmath; \ol\beta \alpha} \,=\, g^\ba(
R^1_{i\ol\beta;\alpha\ol\jmath} +
[R^1_{\alpha\ol\jmath},R^1_{i\ol\beta}])\,=\, -
(\Phi\Phi^*)_{i\ol\jmath} + g^\ba [R^1_{\alpha\ol\jmath},
R^1_{i\ol\beta}]\\
=\,-\Phi_{;i}\Phi^*_{;\ol\jmath}-(\Phi_{\ol\jmath j}- R^1_{i\ol\jmath}
\Phi + \Phi R^2_{i\ol\jmath})\Phi^* + g^\ba [R^1_{\alpha\ol\jmath},
R^1_{i\ol\beta}]
\end{gather*}
so that
$$
(d^*dR_{i\ol\jmath})_1 \,=\, \Phi_{;i}\Phi^*_{\ol\jmath}-
g^\ba[R^1_{\alpha\ol\jmath},R^1_{i\ol\beta}].
$$
In a similar way
$$
(d^*dR_{i\ol\jmath})_2 \,=\, -\Phi^*_{\ol\jmath}\Phi_{;i}-
g^\ba[R^2_{\alpha\ol\jmath},R^2_{i\ol\beta}]\, .
$$
There is a natural (pointwise) inner product
$$
\wt C^1 \times \wt C^1 \,\longrightarrow\, \wt C^0\, .
$$
defined for sections
$(a\, ,b)\,=\, (a\, ,(b_{1\ol\beta}dz^\ol\beta\, ,
b_{2\ol\beta}dz^\ol\beta))$
and
$(a'\, ,b')\,=\,(a'\, ,(b_1'\, ,b_2'))$ of
$$
\cA^{0,0}(Hom(E_2,E_1))\oplus
\cA^{0,1}(End(E_1)\oplus End(E_2))
$$
by

\noindent
    \parbox{\textwidth}{\begin{equation}
(a\, ,b)\cdot(a'\, ,b')\,:=\, ( a {{a'}}^* + g^\ba [b_{1\ol\beta}\, ,
b^{'*}_{1\alpha}]\, , -{{a'}}^* a + g^\ba [b_{2\ol\beta}\, ,
b^{'*}_{2\alpha}]\, .
\end{equation} }

Then the equality

\noindent
    \parbox{\textwidth}{
        \begin{equation}\label{dstardrij}
            d^*dR_{i\ol\jmath}\,=\, \mu_i \cdot \mu^*_{\ol\jmath}
        \end{equation}
    }

\noindent holds.

Now, we are in a position to compute the curvature tensor of the \pw
metric.

We refer to Definition \ref{de:pwherm} for the metric tensor
$G^{VM}_{i\ol\jmath}$.

First, we compute first partial derivatives of the metric tensor.

We claim

\noindent
    \parbox{\textwidth}{
        \begin{equation}\label{Gijk}
            G^{VM}_{i\ol\jmath|k}(s)\,= \,\int_{X\times\{s\}} \tr
(\mu_{i;k}
            \mu^*_{\ol\jmath}) g\, dV
        \end{equation}
    }

\begin{proof}
  We have
$$
\int\tr( \mu_i \mu^*_{\ol\jmath; k}) g \, dV \,=\, \langle \mu_i\, ,
dR_{j\ol k}
    \rangle \,=\, \langle d^* \mu_i\, , R_{j,\ol k}\rangle \, =\, 0
$$
because of \eqref{dstarmu} and \eqref{mudr}.
\end{proof}
At a given point $s_0 \,\in\, S$ we introduce holomorphic
normal coordinates of the second kind, which means that
the K\"ahler form coincides with the constant one
in terms of the coordinate chart, up to order two at $s_0$,
or equivalently, the partial derivatives of the local expression
of the Hermitian metric vanish at $s_0$. From
\eqref{Gijk} it follows that this condition is equivalent to
the condition that the
harmonic projections of all $\mu_{i;k}$ vanish (at $s_0$):

\noindent
    \parbox{\textwidth}{
        \begin{equation}\label{harm}
            H(\mu_{i;k}(s_0))\, =\, 0\, .
        \end{equation}
    }

We denote by $G$ the (abstract) Green's operator. Then \eqref{harm}
and \eqref{dstarmuik} imply that, for $s\,=\,s_0$,
$$
\mu_{i;k}\,=\,Gd^*d\mu_{i;k}\,=\,d^*Gd\mu_{i;k}\, .
$$
Together with
\eqref{dmuik} we get

\noindent
    \parbox{\textwidth}{
        \begin{equation}\label{gmuwedgemu}
           \mu_{i;k}(s_0)\,=\,-d^*G[\mu_i\wedge \mu_k]\, .
        \end{equation}
    }

Now (in terms of normal coordinates)
$$
-R^{VM}_{i\ol\jmath k \ol\ell}\,=\, G^{VM}_{i\ol\jmath|k\ol\ell}\,=\,
\int
\tr(\mu_{i;k\ol\ell}\cdot \mu^*_\ol\jmath ) \, g\, dV +  \int
\tr(\mu_{i;k}\cdot \mu^*_{\ol\jmath;\ol\ell}) \, g\, dV\, .
$$
According to \eqref{gmuwedgemu} the second integral equals
$$
\int \tr\left( G([\mu_i\wedge \mu_k]) \cdot
[\mu^*_{\ol\jmath}\wedge\mu^*_{\ol\ell}] \right)\, g\, dV.
$$
We compute the first integral: It equals
$$
\int \tr(\Phi_{;ik\ol\ell} \cdot \Phi^*_\ol\jmath)\,g\, dV +
\sum_{\nu=1,2}\int \tr(
R^\nu_{i\ol\beta;k\ol\ell}R^\nu_{\alpha\ol\jmath})\,g\, dV\, .
$$
Therefore, we have
$$
I\,=\, I_0+I_1+I_2\, ,
$$
where
$$I \,:=\, \int
\tr(\mu_{i;k\ol\ell}\cdot \mu^*_\ol\jmath ) \, g\, dV\,\, ,\, ~\,
I_0\,:=\, \int \tr(\Phi_{;ik\ol\ell} \cdot \Phi^*_\ol\jmath)\,g\, dV
$$
and
$$
I_i \,:=\, \sum_{\nu=1,2}\int \tr(
R^\nu_{i\ol\beta;k\ol\ell}R^\nu_{\alpha\ol\jmath})\,g\, dV\, ,
$$
where $i\,=\,1\, ,2$.

Now with convention \eqref{conv}
$$
\Phi_{;ik\ol\ell}\,=\, (\Phi_{;\ol\ell i}- [R_{i\ol\ell}\, ,
\Phi])_{;k} - [R_{k\ol\ell}\, ,\Phi_{;i}]
$$
so that
$$
I_0\,=\, -\int
\tr\left(([R_{i\ol\ell}\, ,\Phi_{;k}]+[R_{k\ol\ell}\, ,\Phi_{;i}]+
[R_{i\ol\ell;k}\, , \Phi])\Phi^*_{;\ol\jmath} \right)\,g\, dV.
$$
Next,
$$
R_{i\ol\beta;k\ol\ell}^\nu\,=\,
R_{i\ell;k\ol\beta}^\nu+[R_{k\ol\beta}^\nu\, ,R_{i\ol\ell}^\nu]+
[R_{i\ol\beta}^\nu\, ,R_{k\ol\ell}^\nu]
\; ; \quad \nu\,=\,1\, ,2
$$
We compute the contributions of $R_{i\ell;k\ol\beta}^\nu$ to the
integrals $I_\nu$; $\nu\,=\,1\, ,2$ and get
$$
-\sum_{\nu=1,2} \int \tr(g^\ba R^\nu_{\alpha\ol\beta;\ol\jmath}
R^\nu_{i\ol\ell;k} )\,g\, dV \,=\, \int \tr ( \Phi\Phi^*_{;\ol\jmath}
R^1_{i\ol\ell;k} - \Phi^*_{;\ol\jmath}\Phi R^2_{i\ol\ell;k})\,g\,dV.
$$
These cancel out together with corresponding terms of $I_0$. Hence
\begin{eqnarray*}
I &=& - \int \tr\left( R^1_{i\ol\ell}(-\Phi_{;k}\Phi^*_{;\ol\jmath}
+ g^\ba [R^1_{\alpha\ol\jmath},R^1_{k\ol\beta}]) + R^2_{i\ol\ell}(
\Phi^*_{;\ol\jmath}\Phi_{;k} +
g^\ba[R^2_{\alpha\ol\jmath},R^2_{k\ol\beta} ] ) \right) \, g\, dV\\
&& - \int \tr\left( R^1_{k\ol\ell}(-\Phi_{;i}\Phi^*_{;\ol\jmath} +
g^\ba [R^1_{\alpha\ol\jmath},R^1_{i\ol\beta}]) + R^2_{k\ol\ell}(
\Phi^*_{;\ol\jmath}\Phi_{;i} +
g^\ba[R^2_{\alpha\ol\jmath},R^2_{i\ol\beta} ] ) \right) \, g\, dV.
\end{eqnarray*}
By \eqref{dstardrij} these terms read
$$
I\,=\, \int \tr(\Box R_{k\ol\jmath}\cdot R_{i\ol\ell} )\, g\, dV + \int
\tr(\Box R_{i\ol\jmath}\cdot R_{k\ol\ell} )\, g\, dV
$$
where again $R_{k\ol\jmath}\,=\,
(R_{k\ol\jmath}^1\, ,R_{k\ol\jmath}^2)$.
Since our aim is an expression in terms of the harmonic \ks tensors,
again we use the Green's operator (which is here the inverse
$\Box_0^{-1}$ of the Laplacian, restricted to the space of
differentiable endomorphisms such that the mean trace vanishes). So
$$
I\,=\, -\int\tr\left((\mu_i\cdot \mu_{\ol\jmath}^*) G(\mu_{k}\cdot
\mu_{\ol\ell}^* )\right)\,g\,dV - \int\tr\left((\mu_k\cdot
\mu_{\ol\jmath}^*) G(\mu_{i}\cdot \mu_{\ol\ell}^* )\right)\,g\,dV\, .
$$
We observe that the result of our curvature computation is
independent of the choice of normal coordinates.

\begin{theorem}\label{thm.curv.}
  The curvature tensor of the \pw metric on the moduli space of
  solutions of the coupled vortex equations equals

\noindent \parbox{\textwidth}{
  \begin{eqnarray}
R^{VM}_{i\ol\jmath k\ol\ell}(s) &=& -\int_{X\times \{s\}} \tr\left(
[\mu_i\wedge \mu_k] \cdot G[\mu^*_{\ol\jmath}\wedge\mu^*_{\ol\ell}]
\right)\, g\, dV \nonumber
\\
&& + \int_{X\times \{s\}}\tr\left((\mu_i\cdot \mu_{\ol\jmath}^*)\,
G(\mu_{k}\cdot \mu_{\ol\ell}^* )\right)\,g\,dV + \int_{X\times
\{s\}}\tr\left((\mu_k\cdot \mu_{\ol\jmath}^*)\, G(\mu_{i}\cdot
\mu_{\ol\ell}^* )\right)\,g\,dV
  \end{eqnarray}
  }
\end{theorem}

If $\dim_{\mathbb C} X\,=\,1$, then only the second term in the
expression of $R^{VM}_{i\ol\jmath k\ol\ell}(s)$ given in Theorem
\ref{thm.curv.} is present.

Therefore, Theorem \ref{thm.curv.} has the following corollary:

\begin{corollary}\label{corollary1}
The \pw metric on any moduli space of stable triples over a compact
Riemann surface has semi--positive holomorphic bisectional curvature.
\end{corollary}
Finally, we note that, if the results
\cite{s-t} can be generalized to stable pairs, then the
methods of \cite{AG} and \cite{g-p} would give an
alternative approach to the computation of the curvature of the \pw
metric.

\medskip
\begin{small}
\noindent \textbf{Acknowledgements.}\, Thanks are due to the referee
for helpful comments. The first named author would like to thank the
Philipps-Universit\"at Marburg for hospitality, whereas the second
named author would like to thank the Tata Institute of Fundamental
Research in Mumbai for its hospitality.
\end{small}

\end{document}